\documentclass[12pt]{amsart}

\usepackage[pdftex,
colorlinks, 
bookmarksnumbered, 
bookmarksopen, 
linkcolor=blue, 
citecolor=blue, 
]{hyperref}
\usepackage{eucal}
\usepackage{mathptmx}
\usepackage{graphicx}
\usepackage{amssymb}
\usepackage{amsthm}
\providecommand {\norm}[1] {\lVert#1\rVert}
\providecommand {\bignorm}[1] {\Bigl\lVert#1\Bigr\rVert}
\providecommand {\Bignorm}[1] {\biggl\lVert#1\biggr\rVert}
\providecommand {\abs}[1] {\lvert#1\rvert}

\providecommand {\set}[1]{\lbrace #1 \rbrace}
\providecommand {\bigset}[1]{\Bigl\lbrace #1 \Bigr\rbrace}
\providecommand {\floor}[1]{\lfloor #1 \rfloor}
\providecommand {\besov}[3]{\Lambda^ {#1} _ {#2} ({#3})}

\providecommand {\bessel}[2]{\mathcal{P}_ {#1} ({#2})}
\providecommand {\dadas}[3]{D^ {#1} _ {#2} ({#3})}
\providecommand {\dada}[2][M]{D^{1}_{#1} ({#2})} 

\providecommand {\inject}{\hookrightarrow}
\providecommand {\inv}[1]{{#1}^{-1}}

\newcommand {\one} e
\newcommand {\bn} {\ensuremath{\mathbb{N}}}

\newcommand {\br} {\ensuremath{\mathbb{R}}}
\newcommand {\bz} {\ensuremath{\mathbb{Z}}}
\newcommand {\bt} {\ensuremath{\mathbb{T}}}

\newcommand {\brd} {\ensuremath{\mathbb{R}^d}}

\newcommand {\mA} {\ensuremath{\mathcal{A}}}
\newcommand {\mD} {\ensuremath{\mathcal{D}}}

\newcommand {\mF} {\ensuremath{\mathcal{F}}}
\newcommand {\mJ} {\ensuremath{\mathcal{J}}}
{\newcommand {\mG} {\ensuremath{\mathcal{G}}}

\newcommand {\mO} {\ensuremath{\mathcal{O}}}
\newcommand {\mM} {\ensuremath{\mathcal{M}}}

\newcommand {\mB} {\ensuremath{\mathcal{B}}}

\newcommand {\mC} {\ensuremath{\mathcal{C}}} 
\newcommand {\mP} {\ensuremath{\mathcal{P}}}

\newcommand {\bop}  {{\ensuremath{\mathcal B ( \ell^2)}}}

\DeclareMathOperator{\dd}{\mathrm{D}}

\DeclareMathOperator{\id}{\mathrm{id}}

\newcommand \BS {Banach space}
\newcommand \BA {Banach algebra}
\newcommand \MA {matrix algebra}

\newcommand \odd {off-diagonal decay}

\newcommand \IC {inverse-closed}

\newcommand {\cexp} [1] [{x t}] {\ensuremath{e^{2 \pi i #1}}}
\newcommand \nci {norm-controlled inversion}
\newcommand \Nci {Norm-controlled inversion}

\newcommand { \muleb} [1] {\frac{d {#1}}{#1}}

 \newcommand {\bigo}{{\mathcal{O}}}

\newtheorem{prop}{Proposition} [section]
\newtheorem{cor}[prop]{Corollary}
\newtheorem{thm}[prop]{Theorem}

\newtheorem{lem}[prop]{Lemma}

\theoremstyle{definition}
\newtheorem{defn}[prop]{Definition} 

\theoremstyle{remark}
\newtheorem*{rem}{Remark}
\newtheorem*{rems}{Remarks}
\newtheorem{ex}[prop]{Example}

\setlength{\topmargin}{-8mm}
\setlength{\headheight}{8pt}
\setlength{\textheight}{220mm}  

\setlength{\oddsidemargin}{0pt}
\setlength{\evensidemargin}{0pt}
\setlength{\textwidth}{148 mm}   

\begin{document}
\begin{abstract}
We show that smoothness implies \nci : the smoothness of an element
$a$  in
a Banach algebra with a one-parameter automorphism group is preserved
under inversion, and the norm of the inverse $a^{-1}$  is controlled by
the smoothness of $a$ and by spectral data.  In our context smooth
subalgebras are obtained with the classical constructions of
approximation theory and resemble spaces of differentiable functions,
Besov spaces or Bessel potential spaces. To treat ultra-smoothness, we
resort to Dales-Davie algebras. Furthermore, based on  Baskakov's
work, we derive explicit norm control estimates for infinite matrices
with polynomial off-diagonal decay. This is a 
quantitative version of Jaffard's theorem.

\end{abstract}

\title{Norm-Controlled Inversion in Smooth Banach Algebras, II}
\author{Karlheinz Gr\"ochenig}
\address{Faculty of Mathematics \\
University of Vienna \\
Nordbergstrasse 15 \\
A-1090 Vienna, Austria}
\email{karlheinz.groechenig@univie.ac.at}
\author{Andreas Klotz}
\address{Faculty of Mathematics \\
University of Vienna \\
Nordbergstrasse 15 \\
A-1090 Vienna, Austria}
\email{andreas.klotz@univie.ac.at}

\date{}
\thanks{This research was  
  supported by the  project P22746-N13  of the
Austrian Science Foundation (FWF)}%
\maketitle

\section{Introduction}

Norm control refers to the phenomenon that in a Banach algebra the
norm of an invertible element can be controlled solely  by its smoothness and
by   spectral data. 

The prototype of  norm-controlled
inversion occurs already for differentiable functions.   Let $ C^1(\bt )$ be the algebra of continuously differentiable functions
on the torus $\bt $ and $ C(\bt )$ be the algebra of all
continuous functions on the torus with the norms $\norm{f}_C =
\norm{f}_\infty = \max
_{t\in \bt } |f(t)|$ and $ \norm{f}_{C^1} = \norm{f}_\infty +
\norm{f'}_\infty $. The quotient rule $(1/f)' = -
f' / f^2$ leads to an obvious estimate for  the $C^1$-norm of  $1/f$,
namely,
\begin{equation}
  \label{eq:sont}
\norm{\tfrac{1}{f}}_{C^1} \leq \norm{\tfrac{1}{f}}_{\infty} + \norm{f'}_\infty
  \norm{\tfrac{1}{f}}_\infty ^2 \leq  \Big(  \norm{f}_{C^1}
  \norm{\tfrac{1}{f}}_C + 1\Big) \norm{\tfrac{1}{f}}_C  \, .
\end{equation}
Here $\norm{f}_{C^1}$ is a measure for the smoothness of $f$, whereas
$\delta = \min _{t\in \bt } |f(t)|) $ is the smallest spectral value
of $f$ and thus $\norm{\inv f } _C = 1/\delta  $ measures a significant parameter
of the spectrum of $f$ in $C(\bt )$.    

In general, we  say that a subalgebra $\mA $ of a unital  Banach algebra \mB\ admits \nci\ if
there exists a function $h: \br_+^2 \to \br_+$ that
    satisfies 
\begin{equation} \label{dnci}
\norm{\inv  a}_\mA \leq h( \norm{a}_\mA , \norm{\inv a }_\mB ) \, 
\end{equation}
for all invertible elements of $\mB $. In particular, the 
invertibility  of $a\in \mA $ in the larger algebra $\mB $ implies the
invertibility of $a$ in the smaller algebra $\mA $ \emph{and} also the
control of the norm of the inverse. Since the norm $\|\inv a \|_\mA $
is related to the condition number of $a$ in the larger algebra $\mA
$, \nci\ implies a quantitative estimate for the condition number in
$\mB $. 

The investigation of \nci\ goes back to the work of Baskakov~\cite{Baskakov97} and
Nikolski~\cite{Nikolski99}.  Baskakov  studied the behavior of the off-diagonal
decay of infinite matrices under inversion. Nikolski studied the lack
of norm control for the algebra of absolutely convergent Fourier
series in $C(\bt )$ and showed  that a modest amount of additional smoothness, as
described by weighted absolutely convergent Fourier series, yields
\nci\   in $C(\bt )$~\cite{elfalla98,Nikolski99}. 

Our own work~\cite{grkl10,klotz12,klotz12pre} was motivated by the
desire to develop a systematic 
theory of inverse-closedness and \nci\ and to understand which
properties of a subalgebra imply \nci . 
In ~\cite{GK12} we introduced the abstract definition  \eqref{dnci}
and showed that \nci\ is possible with  an 
extremely weak notion of smoothness. The main result of \cite{GK12} can be
formulated as follows.

\begin{thm} \label{main0}
  Assume that  \mB\ is a $C^*$-algebra and $\mA \subseteq \mB$ is a
  $*$-subalgebra with a common unit    and a differential (semi-)
  norm, i.e., $\norm{\cdot }$
  satisfies  the inequality 
\begin{equation}
\norm{ab}_\mA \leq C(\norm{a}_\mA \norm{b}_\mB+\norm{b}_\mA \norm{a}_\mB)\label{eq:diffsnint}
\end{equation}
for all $a, b \in \mA$.  Then there exist constants $C_1, C_2 > 0$   such that 
  $$
\norm{\inv a } _{\mA } \leq C_1 \norm{a}_{\mA } \norm{\inv a }_{\mB } ^2
\, e^{C_2 \log ^2 ( \norm{a}_{\mA } \norm{\inv a }_{\mB } )} \, .
$$
\end{thm}
Note that in this estimate the dependence of $\|\inv a \|_{\mA} $ on
$\inv \delta = \norm{ \inv a}_{\mB}  $ is
faster than any polynomial, whereas in the motivating example $C^1(\bt
)$ it is just $\delta ^{-2}$. 

In this paper we will impose  more structure on the ambient algebra
$\mB $ and will investigate stronger notions of smoothness in a Banach
algebra. Our goal is to  derive stronger quantitative estimates for 
\nci . On an abstract level we develop a systematic
construction of ``smooth'' subalgebras that admit \nci . On a concrete
level we investigate Banach algebras of infinite matrices and find
explicit quantitative estimates for the off-diagonal decay of the
inverse.

Let us describe the main conclusions  in more detail.

(a) \emph{``Smoothness implies norm controlled inversion.''} We will
derive several theorems that  give a precise meaning to this
meta-theorem. In operator theory, smooth elements in a Banach algebra
are  usually defined by a derivation, quite in analogy to the
definition of the function space $C^n(\bt )$ of $n$-times continuously
differentiable functions on the torus $\bt$. This analogy motivated
the non-commutative approximation theory in ~\cite{grkl10}. One of 
 the main insights  was that the 
standard constructions of smoothness spaces for functions possess a
non-commutative analogue for Banach algebras.   In particular, for a
Banach algebra with a (one parameter) automorphism group one can
define subalgebras that correspond to H\"older-Lipschitz spaces or to
Besov spaces or to Bessel potential spaces. Whereas the main results
of~\cite{grkl10,klotz12} can be paraphrased by saying that ``smoothness'' in a
Besov subalgebra is
preserved under inversion (a qualitative result),  we will derive
quantitative versions, i.e., we will derive explicit results about the
\nci\  in such subalgebras. Typically, the notation of a smooth
subalgebra contains a  smoothness parameter $r$, for
instance, the notation of the  Besov subalgebra  $\mC = \Lambda ^p_r(\mA )$
of a Banach algebra $\mA $  indicates that $\mC $ consists of elements with
smoothness $r$. For such subalgebras we will prove
the \nci\ of the form
\begin{equation}
  \label{eq:i9}
  \norm{\inv a}_\mC \leq C   \norm{\inv a}_\mA ^{r+1}   \norm{a}_\mC ^r
  \, .
\end{equation}
 The precise definitions and formulations are somewhat more
 technical and contained in Theorems~\ref{ncideralg},
 \ref{thm-besov-nci} and~\ref{prop:thm-besspot_baic}. These 
 theorems can be seen as a vast generalization of the quotient
 rule~\eqref{eq:sont}. To our knowledge these results are new even for
 Besov spaces on the torus. 

(b) \emph{Norm-controlled inversion and off-diagonal decay of infinite
matrices.}  Perhaps the most important application of \nci\ is to
Banach algebras of infinite matrices. The relevant algebra is $\bop $
of bounded operators on $\ell ^2(\bz )$, and the smoothness of an
matrix  amounts to its off-diagonal decay. The prototypical result is
Jaffard's theorem~\cite{Jaffard90}:  If $A= (A(k,l))_{k,l\in 
\bz } $ is an infinite matrix with off-diagonal decay 
\begin{equation}
  \label{eq:e1}
|A(k,l)| \leq C (1+|k-l|)^{-r} \qquad k,l\in \bz \, ,\quad r>1,  
\end{equation}
and if $A$ is (boundedly) invertible on $\ell ^2(\bz )$, then there
exists a constant $C'$ such that
$$
|\inv A (k,l) | \leq C' (1+|k-l|)^{-r} \qquad k,l\in \bz \, .
$$
Thus the quality of the off-diagonal decay is preserved. In the
language of Banach algebras Jaffard's theorem states 
that the algebra of matrices with polynomial off-diagonal decay of
order $r$ is inverse-closed in $\bop $.
 This powerful result
had many important applications in wavelet theory~\cite{Jaffard90}, time-frequency analysis~\cite{book}, frame theory~\cite{gro04}, and numerical analysis~\cite{GRS10}.
We also refer to~\cite{Gr10} for a survey of this
topic. We emphasize that Jaffard's theorem shows only the existence of
a constants $C'$ for the inverse matrix and is thus a purely
qualitative result. For numerical purposes, however, 
it is of utmost importance to understand how
big the constant $C'$ for the decay of $\inv A $ can be. In our
language,  we need to solve the problem whether  Jaffard's algebra
admits \nci .

Although polynomial off-diagonal decay is related to the axiomatic
construction of Besov algebras and Bessel algebras, the Banach algebra
defined by the decay conditions~\eqref{eq:e1} cannot be realized
directly within the abstract theory. Therefore neither the
inverse-closedness nor the \nci\ in $\bop $ follow directly from our
theory and require deeper ideas. A proof of the \nci\ of Jaffard's
algebra is implicit in the deep work of
Baskakov~\cite{Baskakov97}. Unfortunately he stopped short of deriving
explicit estimates, and thus his fundamental result has not received
as much attention as it deserves.  In this paper we make the extra
step and prove the following  explicit quantitative version of Jaffard's theorem.

\begin{thm}[Quantitative version of Jaffard's theorem] \label{thm:in0}
Assume that the infinite matrix $A$ is (boundedly) invertible on $\ell
^2(\bz )$ and that it possesses off-diagonal decay of order $r>1$ as
in~\eqref{eq:e1} with minimal constant $C$. Then the inverse matrix possesses the off-diagonal
decay 
$$
    |A^{-1} (k,l)|   \leq C_r \, C^{2r+2 +\frac{2}{r-1}} \norm{\inv
   A}_{\bop}^{ 2 r+3 + \frac{2}{r-1}}  \, ,
$$    
where $C_r$ is a  constant depending only on
$r>1$.
\end{thm}

Again, the norm of the inverse depends only on the smoothness, which
in this case corresponds to the off-diagonal decay of $A$, and
spectral information of $A$ on $\ell ^2$. The formulation is similar
to the theorem of Demko, Moss, and Smith on the exponential  decay of
the inverse of a banded matrix~\cite{Demko84}, but the proof of Theorem~\ref{thm:in0} is much more
difficult. A weaker form of \nci\ was studied in~\cite{romero}. 


(c) In a sense much of the theory of non-commutative smoothness spaces
amounts to studying approximation theory with weights of
polynomial growth.  A corresponding theory for weights of
super-polynomial growth is much less developed. In this paper we
choose the  notion of Dales-Davie algebras to study subalgebras of a
given Banach algebra that are associated to a weight of
super-polynomial growth. The elements in the Dales-Davie algebras
correspond to ultra-smooth elements. Our main result states that a
Dales-Davie algebra admits \nci .

The paper is organized as follows: In Section~2 
 we recall several
possible concepts of smoothness in a Banach algebra and define Besov
 algebras, Bessel algebras, and Dales-Davie algebras. We prove
 the explicit \nci\ for each of these classes of subalgebras. In
 Section~3 
 we apply the abstract results to concrete Banach
 algebras of matrices and derive a quantitative version of Jaffard's
 theorem. The appendix contains the proofs of the higher-order
 quotient rules for derivations and difference operators.

\section{Norm Controlled Inversion: Improved Estimates by Quotient Rules}
\label{sec:quotient}
Let $\mA \subseteq \mB$ be \BA s with common unit element.
 We say that 
    $\mA$ admits \nci\ in \mB, if there is a function $h: \br_+^2 \to \br_+ $ that
    satisfies 
\begin{equation*}
\norm{ \inv a }_\mA \leq h( \norm{a}_\mA , \norm{ \inv a }_\mB ) \, .
\end{equation*}
If we only know that  $a\in \mA $ and $\inv a \in \mB $ implies $\inv
a\in \mA $, we say that $\mA $ is \IC\ in $\mB$.  

 The control function $h$ in this  definition  is not unique
 and it is often convenient to describe \nci\ in a  slightly different
 way.

\begin{lem}[\cite{GK12}]
  Let $\mA \subseteq \mB$ be \BA s with common unit element. Then
  \mA\ admits norm control, if and only if  
there exists a function $\phi  : (0,1) \to \br_+$, 
  such that,   for $a \in \mA$, 
  \begin{equation}
    \label{eq:3}
   \norm{a}_\mA \leq 1 \qquad \text{ and } \qquad  \norm{ \inv a }_\mB \leq 1/\delta
  \end{equation}
implies that $\inv a  \in \mA $  and 
$$\norm{\inv a }_\mA \leq \phi (\delta ) \, .
$$
 Explicitly, $\phi (\delta ) $ can be chosen to be 
\begin{equation} \label{eq:visconst}
  \phi(\delta)=\sup \set{\norm{\inv a }_\mA \colon \norm{a}_\mA \leq
    1, \norm{\inv a }_\mB \leq 1/\delta} \,.
\end{equation}
The norm control function $h$ is then 
\begin{equation}
  \label{eq:c1}
  h(\norm{a}_\mA , \norm{\inv a} _\mB) =  \frac{1}{\norm{a} _\mA } \phi
  \Big( \frac{1}{\norm{a}_\mA  \norm{\inv a} _\mB}\Big)   \, .
\end{equation}
\end{lem}

\subsection{Domains of Derivations Admit Norm Controlled Inversion}
\label{sec:domder}
  An unbounded  \emph{derivation} $\dd$ on a \BA\ \mA\ is a \emph{closed} linear
  mapping $\dd \colon \mD \to \mA$ with  \emph{domain}
  $\mD=\mD(\dd)=\mD(\dd,\mA)\subseteq \mA$   that satisfies the Leibniz rule
  \begin{equation}
    \label{eq:derivation}
    \dd(ab) =a \dd(b) + \dd(a) b \qquad \text{for all} \,\,
    a,b \in \mD. 
  \end{equation}
We always choose   a maximal domain and endow it with the 
  seminorm $\abs{a}_{\mD}=\norm{\dd(a)}_\mA$ and the graph norm $\norm{a}_{\mD}=\norm{a}_\mA+\abs{a}_\mD$.
Then $\mD$ is again a \BA.
\begin{rem}
It should be observed that we do not assume the domain of $\mD$ to be dense in \mA.  This is important for some of the \MA s that serve as examples in Section~\ref{sec:norm-contr-invers-3}.
\end{rem}
  If \mA\ is a $*$-algebra, we assume that the derivation and
  the domain are symmetric, i.e., $\mD=\mD^*$ and $\dd(a^*)
  =\dd(a)^*$ for all $a \in \mD$. 

For derivations  the following version of the quotient rule is valid. 
  \begin{prop}[{\cite[3.4]{grkl10}, \cite{bratteli96,MR1221047}}]
    Let \mA\ be a symmetric \BA\  and
  $\dd$ be  a symmetric  derivation on
  \mA. 
  If $\one \in \mD (\dd )$, then $\mD (\dd )$
  is \IC\ in \mA . Moreover,  the quotient rule
    \begin{equation}
  \dd(a^{-1})=- a^{-1}\dd(a)a^{-1} \label{eq:quot1}
  \end{equation}
  is valid. 
  \end{prop}

The quotient rule yields an immediate statement about norm control in
$\mD  (\dd )$. 

\begin{cor}
Assume that $\mA$ is a symmetric \BA\ with a closed derivation
$\dd$ and $\one \in \mD (\dd )$.  Then $\mD(\dd,\mA)$ admits
norm control in $\mA $. Precisely, if $a \in \mD(\dd,\mA)$ and $a$ is invertible in \mA, then
\[
\abs{\inv a}_\mD \leq \norm{\inv{a}}_\mA^2 \abs{a}_\mD
\] 
In particular, if $\abs{a}_\mD \leq 1$ and $ \norm{\inv{a}}_\mA\leq 1/\delta$, then 
\[
\abs{\inv a}_\mD \leq 1/\delta^2 \,.
\]
If $\norm{a}_\mD \leq 1$ and $ \norm{\inv{a}}_\mA\leq 1/\delta$, then 
\[
\norm{\inv a}_\mD \leq  1/ \delta+ 1/\delta^2 \,.
\]
\end{cor}

\subsubsection*{Norm Controlled Inversion in $\mD(\dd ^k )$}
\label{sec:norm-contr-invers-2}

We next formulate higher order versions of this statement.

The domain of $\dd^k$ is
defined by induction as 
$
\mD (\dd^k ) = \mD (\dd, \mD (\dd^{k-1})) \, .\label{eq:higherder}
$
  In analogy to $C^k(\brd)$ we equip $\mD(\dd^k)$ with the seminorm
\[
\abs{a}_{\mD(\dd^k)} =  \sum_{m=1}^k \frac{\norm{\dd^m(a)}_\mA}{\ m!}
\]
and set $\norm{a}_{\mD(\dd^k)}=\norm{a}_\mA+\abs{a}_{\mD(\dd^k)}$. 
Since $\dd$ is assumed to be a closed operator on $\mA $,  $\dd^k$ is
a closed operator on $\mD (\dd ^k 
)$ (see, e.g. \cite{grkl10}). 
If \mA\ is symmetric and  $\one \in \mD(\dd)$
    then $\mD(\dd ^k )$ is \IC\ in \mA. 
~\cite[3.7]{grkl10}.

We will show that $\mD(\dd^k)$ admits \nci\ in \mA\ for all $k \in
\bn$. 

To extend ~\eqref{eq:quot1}, 
we need the iterated quotient rule.

\begin{lem}\label{invnormderiv}
Let $\dd$ be a derivation on $\mA $ and $k\in \bn $. 
If $\inv a \in \mD (\dd ^k)$, then 
    \begin{equation}\label{eq:itquot}
    \dd^k(\inv a)=\sum_{m=1}^k(-1)^m \sum_{\substack{k_1+ \dotsb k_m=k \\ k_j \geq 1}} \binom{k}{k_1, \dotsc,k_m}
    \Bigl(\prod_{i=1}^m [\inv a \dd^{k_i}(a)]\Bigr) \inv a \,.
  \end{equation}
\end{lem}
Since we do not know a precise reference for this rule, we provide a
proof in the appendix. A more general formula is contained  in
\cite{klotz12pre}. 

\begin{prop}\label{ncideralg}
  If \mA\ is a symmetric \BA\ with an unbounded,  closed  derivation $\dd$ and $k
  \in \bn$, then $\mD(\dd^k)$ admits \nci\ in \mA , and  
  \begin{equation}
\abs{\inv a}_{\mD(\dd^k)} \leq \norm{\inv{a}}_\mA^2
\abs{a}_{\mD(\dd^k)} \max \Big( k,  \frac{\norm{\inv a}_\mA^k
  \abs{a}_{\mD(\dd^k)}^k-1}{ \norm{\inv
    a}_\mA\abs{a}_{\mD(\dd^k)}-1} \Big) \, .
\label{eq:ncideri}
\end{equation}
If $\norm{ \inv a}_\mA \abs{a}_{\mD(\dd^k)} > \max(2,k) $  then this expression can be simplified to
  \begin{equation}
    \label{eq:2}
   \abs{ \inv a }_{\mD(\dd^k)} \leq 2 \norm{\inv a }_\mA ^{k+1}
   \abs{a} _{\mD(\dd^k)}^{k}  \, .
  \end{equation}
Asymptotically, if $|a|_{\mD(\dd^k)} \leq 1$ and $\norm{\inv a }_\mA
\leq \inv \delta   $, then 
$$
|\inv a |_{\mD(\dd^k)} \leq 2 \delta ^{-k-1} \, .
$$
\end{prop}
\begin{proof}
We use the iterated quotient rule (Lemma~\ref{invnormderiv}). Then 
  \begin{align*}
    \abs{\inv a}_{\mD(\dd^k)}&=\sum_{m=1}^k \frac{\norm{\dd^m(\inv a)}_\mA}{m!} \\
    & \leq \sum_{m=1}^k\frac{1}{m!}\sum_{n=1}^m\sum_{\substack{m_1+
        \dotsb + m_n=m \\ m_j \geq 1}} \binom{m}{m_1, \dotsc,m_n}
    \bignorm{\Bigl(\prod_{i=1}^n \inv a \dd^{m_i}(a) \Bigr) \, \inv
      a}_\mA \\ 
    & \leq \sum_{m=1}^k \sum_{n=1}^m \norm{\inv
      a}_\mA^{n+1}\sum_{\substack{m_1+ \dotsb + m_n=m \\ m_j \geq 1}}\,
    \prod_{i=1}^n \frac{\norm{\dd^{m_i}(a) }_\mA}{m_i !} \\ 
    & =    \sum_{n=1}^k \norm{\inv a}_\mA^{n+1} \sum_{m=n}^k \,
    \sum_{\substack{m_1+ \dotsb  + m_n=m \\ m_j \geq 1}} \, \prod_{i=1}^n
    \frac{\norm{\dd^{m_i}(a) }_\mA}{m_i !} \\
&\leq \sum_{n=1}^k \norm{\inv a}_\mA^{n+1} \sum_{(m_1, \dots , m_n)
  \in \{1,\dots , k\}^n}  \,
 \prod_{i=1}^n
    \frac{\norm{\dd^{m_i}(a) }_\mA}{m_i !}
  \end{align*}
The summation over $(m_1, \dots , m_n)
  \in \{1,\dots , k\}^n$ yields 
$
(\sum _{m=1}^k \frac{\norm{\dd^{m}(a) }_\mA}{m !})^n =
|a|_{\mD(\dd^k)} ^n$, and so 
\begin{align*}
  \abs{\inv a}_{\mD(\dd^k)} &\leq    
    \sum_{n=1}^k \norm{\inv a}_\mA^{n+1} \abs{a}_{\mD(\dd^k)} ^n \\
    & =   \norm{\inv a}_\mA^{2} \abs{a}_{\mD(\dd^k)}  \sum_{n=0}^{k-1} \norm{\inv a}_\mA^{n} \abs{a}_{\mD(\dd^k)} ^n \\
    & \leq \norm{\inv{a}}_\mA^2\abs{a}_{\mD(\dd^k)} \max \Big( k,  \frac{\norm{\inv a}_\mA^k  \abs{a}_{\mD(\dd^k)}^k-1}{ \norm{\inv a}_\mA\abs{a}_{\mD(\dd^k)}-1} \Big) \, ,
  \end{align*}
  and this is what we wanted to show. 
The simplified expression follows by writing $\frac{q^k-1}{q-1} =
\frac{q^{k-1}-1/q}{1-1/q} \leq q^{k-1}$ for $q=\norm{\inv
  a}_\mA\abs{a}_{\mD(\dd^k)}\geq \max (2,k)$. 
\end{proof}
\begin{rem}
  The proposition also follows as a  corollary of the technical  Theorem~\ref{thm-dd-nci}.
\end{rem}




\subsection{Automorphism groups and norm-controlled inversion of Besov algebras}
\label{sec:besov-algebras}
For us a (one-parameter) \emph{automorphism group}
 on \mA\    is a group  of  uniformly bounded  automorphisms $\Psi=\set{\psi_t}_{t \in \br}$ of
\mA . This means that (a)  each $\psi _t: \mA \to \mA $ is a Banach algebra
automorphism of $\mA $, (b)  
$
  \psi_s \psi_t=\psi_{s+t} \quad \text{for all} \quad s,t \in \br
$
and  (c) the $\psi _t$ are uniformly bounded, i.e. 
$
M_\Psi= \sup_{t\in \br}\norm{\psi_t}_{\mA \to \mA} < \infty \, .
$
If \mA\ is a $*$-algebra we assume that $\Psi$ consists of $*$-automorphisms.  

The subalgebra $C(\mA )$ of continuous elements consists of all  elements $a
\in \mA$ satisfying $\lim_{t \to 0}\psi_t(a)=a$.

The \emph{generator} $\dd$ of $\Psi$ is defined pointwise by
$
  \dd  (a) = \lim_{h \to 0} \frac{\psi_{h}(a)-a }{h}
$. 
It is a closed derivation, and
the domain $\mD(\dd, \mA)$ of $\dd$ is the set of all $a \in \mA$ for which this limit exists. %
If \mA\ is a $*$-algebra, then $\dd$ is symmetric. See, e.g.~\cite{Butzer67}.


By means of  the automorphism group $\Psi$ we can  define the  Besov
spaces on the \BA\ \mA\ in complete analogy to the classical Besov
spaces on $\br $. Of the many equivalent definitions  we choose the
following one, see, e.g.,~\cite{Butzer67}. Let $\Delta _t a = \psi
_t(a)  - a$ be the  difference operator and   $\Delta_t^k=  (\psi_t
-\id)^k$ be the higher differences. Assume that  $1\leq p < \infty $,
$r>0$, and define $k=\lceil r \rceil $ as the smallest integer greater than or equal to $r$.  Then the Besov space $\besov p r
\mA $   consists of all $a\in \mA $ for which   the expression
\begin{equation}\label{eq:besovnorm}
\norm{a}_{\besov p r\mA} =  \norm{a}_\mA + \Bigl(\int_{\br} (\abs{t}^{-r}\norm{\Delta^k_t a}_\mA)^p \frac{dt }{\abs t}\Bigr)^{1/p} = \norm{a}_\mA + \abs{a}_{\besov p r\mA}
\end{equation}
is finite. For $p =  \infty $, we use the supremum, as usual. Then 
\eqref{eq:besovnorm} is  a norm on $\besov p r \mA$ and $\besov p r
\mA $ is a Banach space. 
If  $k=\lceil r \rceil $ is replaced  by any  integer $k_1 > r$ in the preceding
definition, then one obtains  an equivalent norm on  $\besov p r \mA$. 

The main properties concerning the 
algebraic structure  were derived  in~\cite{klotz12}. 
\begin{prop} [{\cite[Thm.~3.7, 3.8]{klotz12}}] \label{prop:baalg}
Let $\mA $ be a \BA\ with a uniformly bounded automorphism group $\Psi
$, $1\leq  p \leq \infty $, and $r>0$.  

(i) Then   $\besov p r \mA$ is a Banach subalgebra of $\mA $. 

(ii) $\Psi$ is an automorphism group on $\besov p r \mA$, and $\norm{\psi_t (a)}_{\besov p r \mA} \leq M_\Psi \norm{ a}_{\besov p r \mA}$.

(iii)  $\besov p r \mA$ is   \IC\ in \mA . 

(iv) \label{thm-reiteration}  Reiteration theorem: If $1 \leq p,q \leq \infty$ and $r,s >0$ then
  \begin{equation}\label{eq:48}
    \Lambda^q_s(\Lambda^p_r(\mA))=\Lambda^q_{r+s}(\mA) \,.
  \end{equation}
 \end{prop}

Like derivations, difference operators obey a product rule and a
quotient rule. We may therefore expect that the Besov algebras also
admit some form of norm control. A first and easy result follows with
the reiteration theorem.

\begin{prop}\label{prop:besov-nci-basic}
Let $1\leq p \leq \infty $ and $r>0$. Then $\besov p r \mA $ admits
norm control, and  
\[
\norm{\inv a}_{\besov p r \mA} \leq C_{p,r} \norm{\inv{a}}_\mA^{2^{\floor r +1}} \norm{a}_{\besov p r \mA}^{2^{\floor r}} 
\]  
In particular, if   $\norm{a}_{\besov p r \mA} \leq 1$ and $\norm{\inv a}_\mA \leq 1/\delta$, then
  \begin{equation}
    \label{eq:besovncit}
    \norm{\inv a}_{\besov p r \mA} \leq C_r \delta^{-2^r}
  \end{equation}
\end{prop}
\begin{proof}
(a) Let $0<r<1$.   The appropriate  quotient rule for the difference
operator  is given by 
  \begin{equation} \label{eq:invderiv}
    \Delta_t(a^{-1})=-\psi_t(a^{-1}) \; \Delta_t(a) \; a^{-1} \, , 
  \end{equation}
  after  a straightforward computation.   Applying  the Besov-norm on
  both sides, we obtain the norm control
  \begin{equation}
    \label{eq:ncbesov}
    \abs{\inv a}_{\besov p r \mA} \leq M_\Psi \norm{\inv{a}}_\mA^2
    \abs{a}_{\besov p r \mA} \, .
  \end{equation}
Using $\norm{\inv a}_\mA \norm{a}_{\besov p r \mA} \geq \norm{\inv
  a}_\mA \norm{a}_{ \mA} \geq  1$,  the last relation yields 
\begin{equation}
  \label{eq:{ncbesov1}}
  \norm{\inv a}_{\besov p r \mA} \leq 2 M_\Psi \norm{\inv{a}}_\mA^2
    \norm{a}_{\besov p r \mA} \, .
\end{equation}

(b) If $r\geq 1$, set $N= \floor r + 1$ and $s =r/N = r/(\floor r +1)
< 1$. Now we apply the reiteration theorem $\besov p {Ns} \mA = \besov p
s { \besov p {(N-1)s} \mA }$, part (ii) of Prop.~\ref{prop:baalg} and step (a)  repeatedly and obtain 
\begin{align*}
  \norm{\inv a}_{\besov p r \mA } & \leq 2 M_\Psi \norm{\inv a}^2_{\besov p {(N-1)s} \mA }
  \, \norm{a}_{\besov p {Ns} \mA } \\
&\leq 2M_\Psi \Big( (2M_\Psi \norm{\inv a }_{\besov p {(N-2)s} \mA }
\norm{a}_{\besov p {(N-1)s} \mA } \Big)^2 \norm{a}_{\besov p {r} \mA }
\\
&\leq C_1 (2M_\Psi )^3 \, \norm{\inv a }_{\besov p {(N-2)s} \mA }^4\,
\norm{a}_{\besov p {r} \mA } ^3 \\
& \leq \dots  \leq C_{N} (2M_\Psi )^{2^N-1} \, \norm{\inv a}_\mA ^{2^N}
\norm{a}_{\besov p r \mA}^{2^N-1} 
 \, . 
\end{align*}
The intermediate constants $C_1, \dots , C_N$  come from the embedding of $\besov p
{(N-l)s} \mA $ in $\besov p {r} \mA $. In the final step we have used
part (a) of the proof. 
\end{proof}
A much better estimate can be derived with an iterated quotient rule
 for higher differences that is  similar to Lemma~\ref{invnormderiv}. 

\begin{lem} \label{undef}
If  $a\in \mA $ is invertible, then 
  \begin{equation}
    \label{eq:15}
    \Delta_t^k(\inv a)=\psi_{kt}(\inv a) \sum_{m=1}^k(-1)^m 
\sum_{\substack{k_1+\dotsc k_m=k\\ k_j\geq1}} \binom{k}{k_1,\dotsc,k_m} \prod_{j=1}^m \psi_{(k -\sum_{l=1}^j k_l)t}\bigl((\Delta^{k_j}a) \inv a \bigr)
  \end{equation}
\end{lem}
The proof is purely algebraic and is given  in the  appendix. Using
Lemma~\ref{undef}, we obtain the following improved version of \nci\ in
$\besov p r \mA $. 
\begin{thm}\label{thm-besov-nci}
Let $\mA $ be a \BA\ with a one-parameter  automorphism group
$\Psi$. Then the  Besov algebra $\besov p  r \mA$ with  $1\leq p \leq \infty $, $r> 0$,  admits the following norm control
in $\mA $:
\begin{equation}
  \label{eq:sep2}
  |\inv a |_{\besov p  r \mA} \leq C   \norm{\inv a}_\mA^2
  \norm{a}_{\besov p r \mA} \frac{ \norm{\inv a}_\mA^{ \lfloor r \rfloor   +1}
    \norm{a}_{\besov p {r} \mA }^{\lfloor r \rfloor   +1} -1}{\norm{\inv a}_\mA
    \norm{a}_{\besov p {r} \mA }-1}  \, ,
\end{equation}
with  a constant $C>0$ depending on $r$ and $p$. 

In particular,  if $\norm{a}_{\besov p  r \mA} \leq 1$, and $\norm{\inv a}_\mA \leq 1/\delta$, then 
  \[
  \norm{\inv a}_{\besov p r \mA} \leq C \delta^{-\floor r-2}
  \]

\end{thm}

\begin{proof}
The iterated quotient rule  \eqref{eq:15} implies that, for every  integer $l > r$, in
  particular $k= \lfloor r \rfloor  +1$, 
  \begin{equation*}
   \norm{ \Delta_t^k(\inv a)}_\mA\leq M_\Psi ^{k+1} \norm{\inv a}_\mA
   \sum_{m=1}^k\norm{\inv a}_\mA^m \sum_{\substack{k_1+\dotsc k_m=k\\
       k_j\geq1}} \binom{k}{k_1,\dotsc,k_m} \prod_{j=1}^m{
     \norm{\Delta_t^{k_j}a}_\mA} \, .
  \end{equation*}
Now choose positive parameters $r_j$  such that
$\sum_{j=1}^m r_j=r$. Then 
 \begin{equation} \label{eq:sep4}
  \abs{t}^{-r} \norm{ \Delta_t^k(\inv a)}_\mA\leq C_k \norm{\inv a}_\mA \sum_{m=1}^k \norm{\inv a}_\mA^m \sum_{\substack{k_1+\dotsc k_m=k\\ k_j\geq1}} \binom{k}{k_1,\dotsc,k_m} \prod_{j=1}^m \frac{ \norm{\Delta_t^{k_j}a}_\mA} {\abs {t}^{r_j}} \,,
  \end{equation}
Note that $\abs {a}_{\besov p r \mA} = \bignorm{\abs{t}^{-r}
  \norm{\Delta ^k a}_\mA}_{L^p(\mathbb{R}, dt/t)}$. Taking the 
$L^p(\mathbb{R}, dt/t)$-norms on both sides of \eqref{eq:sep4}, we
obtain 
\begin{align*}
  \abs{\inv a&}_{\besov p r \mA} \leq C_k \norm{\inv a}_\mA
  \sum_{m=1}^k \norm{\inv a}_\mA^m \sum_{\substack{k_1+\dotsc k_m=k\\
      k_j\geq1}} \binom{k}{k_1,\dotsc,k_m}  
\Bignorm{\prod_{j=1}^m \frac{ \norm{\Delta_t^{k_j}a}_\mA} {\abs t
    ^{r_j}} }_{L^p(\br , dt/t)} \\
& \leq  C_k \norm{\inv a}_\mA  \sum_{m=1}^k \norm{\inv a}_\mA^m
\sum_{\substack{k_1+\dotsc k_m=k\\ k_j\geq1}}
\binom{k}{k_1,\dotsc,k_m} 
\Bignorm{ \frac{ \norm{\Delta_t^{k_1}a}_\mA} {\abs{t}^{r(k_1)}
  }}_{L^p(\br , dt/t)} \sup _{t\neq 0} \prod_{j=2}^m \frac{
  \norm{\Delta_t^{k_j}a}_\mA} {\abs t^{r_j}}\\ 
 & \leq  C_k \norm{\inv a}_\mA  \sum_{m=1}^k \norm{\inv a}_\mA^m \sum_{\substack{k_1+\dotsc k_m=k\\ k_j\geq1}}  \binom{k}{k_1,\dotsc,k_m}
 \abs{ a }_{\besov p {r(k_1)} \mA} \prod_{j=2}^m \abs{a}_{\besov
   \infty {r_j} \mA } \, . 
  \end{align*}
As $\abs{a}_{\besov p r \mA} \leq C \norm{a}_{\besov q s \mA}$ for
$r\leq s$ and $q \leq p$ (see, e.g., ~\cite{Butzer67}),   we can estimate the product in this expression by
\begin{align*}
   \abs{\inv a}_{\besov p r \mA} &\leq  C_k \norm{\inv a}_\mA  \sum_{m=1}^k \norm{\inv a}_\mA^m \sum_{\substack{k_1+\dotsc k_m=k\\ k_j\geq1}}  \binom{k}{k_1,\dotsc,k_m} \norm{a}_{\besov p {r} \mA }^m \\
                              & \leq C'_k \norm{\inv a}_\mA \sum_{m=1}^k \norm{\inv a}_\mA^m  \norm{a}_{\besov p {r} \mA }^m \\
                          & = C'_k    \norm{\inv a}_\mA^2
                          \norm{a}_{\besov p r \mA} \frac{ \norm{\inv
                              a}_\mA^k  \norm{a}_{\besov p {r} \mA }^k
                            -1}{\norm{\inv a}_\mA  \norm{a}_{\besov p
                              {r} \mA }-1}  \, .
\end{align*}
If  $\norm{a}_{\besov p r \mA}\leq 1$ and $\norm{\inv a}_\mA\leq \delta$,
the resulting  asymptotic estimate is
\begin{align*}
  \abs{\inv a}_{\besov p r \mA} &\leq  C \inv \delta \sum_{m=1}^k
  \delta^{-m} = C \delta^{-2} \frac{\delta^{-k}-1}{\inv \delta  -1} \\
 &\sim C \delta^{-k-1} = C \delta^{-\floor r -2}\quad \text{for} \quad \delta \to 0 \,.
\end{align*}
 \end{proof}
 \begin{rem}
A multivariate version of Theorem~\ref{thm-besov-nci}  remains correct, if we start with  a commutative
$d$-parameter automorphism group. One can adapt the definition of
Besov algebras by  using multivariate differences. 
 \end{rem}
\subsection{Bessel Algebras}
\label{sec:bessel-algebras}
For the definition of Bessel algebras we use an extension of the
automorphism group $\Psi $ from $\br $ to $\mM (\br )$, the algebra
of bounded Borel measures. If $\mu $ is a bounded measure and $a\in
C(\mA )$, then  
the \emph{convolution} of $\mu $ with $a$ is defined as 
   \begin{equation}\label{eq:module}
    \mu * a = \int_{\br }\psi_{-t}(a) d\mu(t). 
  \end{equation}
 This action is a generalization of the convolution of functions and
 satisfies similar properties, in particular, 
\[ \norm{\mu * a}_\mA \leq M_\Psi  \norm{\mu}_{\mM(\br)}\, \norm{a}_\mA  \,. \]

The convolution can be defined in the more general case that $\mA$ is
a $C_w$ group, see~\cite{klotz12} for the definition. In the following
we will assume that the convolution is well-defined and will not dwell on
these technicalities. 

Bessel potentials are an alternative family of smoothness spaces on
$\brd $ and include the standard Sobolev spaces.   For the definition
on a Banach algebra with an automorphism group, let 
$\mG_r$ be  the \emph{Bessel kernel} that is defined by  its Fourier
transform as 
\[
\mF \mG_r( \omega)= (1+\abs{2 \pi \omega}^2)^{-r/2} \, , \quad r >0.
\]
\begin{defn}\label{defn:bessel}
  Let \mA\ be a \BS\ and $\Psi$ a $C_w$- group~\cite{klotz12} acting on
  \mA\ (in particular, this is   the case if  $C(\mA)=\mA$). The \emph{Bessel potential space} of order $r>0$ is
  \[
  \mP_r(\mA)= \mG_r * \mA = \set{a \in \mA \colon a= \mG_r * y \text{ for some } y \in \mA }
  \]
  with the norm
  \[
  \norm{a}_{\bessel r \mA}=\norm{\mG_r * y}_{\bessel r \mA} =\norm{y}_\mA.
  \]
\end{defn}
We collect the main properties of the Bessel potential spaces on $\mA
$ from~\cite{klotz12}. 

\begin{lem}\label{algpropbessel}
Let  \mA\ be  a \BA\ with $C_w$-group $\Psi$ and $r>0$. 

(i) Then $\mP_r(\mA) $ is a \BA . 

(ii) $\Psi$ is an automorphism group on $\bessel  r \mA$, and
$\norm{\psi_t (a)}_{\bessel  r \mA} \leq M_\Psi \norm{ a}_{\mP _ r
 ( \mA  )}$. 

(iii) $\mP_r(\mA) $ is \IC\ in $\mA $. 
  
(iv) The following embeddings hold:
\begin{equation}\label{eq:beinbe}
\besov 1 r \mA \inject \bessel r \mA \inject \besov \infty r \mA \, . 
\end{equation}

(v) Reiteration theorem: for all $r,s >0$
\begin{equation}
  \mP_r(\mP_s(\mA)) = \mP_{r+s}(\mA) \,.\label{eq:42}
  \end{equation}

(v)  Characterization by a hyper-singular integral:  if  $0 < r <2 $,
then 
\[
\norm{a}_\mA + \sup_{\epsilon > 0}\bignorm{ \int_{\epsilon \leq \abs t \leq 1} \frac{\Delta_t(a)} { \abs{t}^{r}}\muleb{t}}_\mA.
\]
is equivalent to the norm $\norm{a}_{\bessel r \mA} $.  
\end{lem}

In contrast to the Besov algebras, the original definition of $\bessel
r \mA $ does not contain  derivations or difference operators  that
satisfy a quotient rule. Nevertheless, using the characterization by a
hyper-singular integral,  we can prove norm control for $\mP _r(\mA )$.  
\begin{thm}\label{prop:thm-besspot_baic}
   If \mA\ is a \BA\ with a $C_w$-group $\Psi$, then
 the Bessel potential space $\bessel r  \mA$  admits \nci\ in \mA. For
 $0<r<1$, the norm control is given explicitly by 
$$
  \norm{\inv a}_{\bessel r \mA} \leq C_r \, \norm{\inv a}_\mA^3 \, 
  \norm{a}_{\bessel r \mA}^2 \, .
$$
     \end{thm}
\begin{proof}
Assume first that $0<r<1$. If $a \in \bessel  r \mA$ , then we want to
estimate  the integral in Lemma~\ref{algpropbessel}(v). 
We rewrite  the quotient rule \eqref{eq:invderiv} as 
\begin{equation*}
  \Delta_t(a^{-1})=-\Delta_t(\inv{a})\Delta_t(a) \inv{a} - \inv a \Delta_t(a) \inv a \,,
\end{equation*}
and so
\begin{multline}\label{eq:besselic}
\lefteqn 
{\bignorm{\int_{\epsilon \leq \abs t  \leq 1}  \frac{\Delta_t(\inv a)}{\abs{t}^{r}}\muleb{t}}_\mA
  \leq }
 \\
  \bignorm{\int_{\epsilon \leq \abs t  \leq 1}  \frac{\Delta_t(\inv a) \Delta_t(a) \inv a}{\abs{t}^{r}}\muleb{t}}_\mA  +   \bignorm{\int_{\epsilon \leq \abs t  \leq 1}  \frac{\inv a \Delta_t( a) \inv a}{\abs{t}^{r}}\muleb{t}}_\mA   \,.
\end{multline}     
As $a \in \besov \infty r \mA$ by \eqref{eq:beinbe}, we have 
$\norm{\Delta_t( a)}_\mA \leq \abs{t}^r  \norm{ a}_{\besov \infty r
  \mA}$. Since  $\Lambda _r ^\infty (\mA )$ is \IC\ in  $\mA $, $\inv
a \in \besov \infty r \mA$ as well, and we  may apply the    quotient rule
\eqref{eq:ncbesov}. 
Using these estimates, the first term on the right hand side of
\eqref{eq:besselic} is dominated by 
\begin{align*}
\int_{\epsilon \leq \abs t  \leq 1} \frac{\norm{\Delta_t(\inv a)}_\mA \norm{\Delta_t a}_\mA \norm{\inv a}_\mA}{\abs t  ^r} \muleb{t} &\leq 
\norm{\inv a}_{\besov \infty r \mA} \norm{a}_{\besov \infty r
  \mA}\norm{\inv a}_{ \mA} \int _\epsilon ^1 |t|^r dt/t \\ 
&\leq C_r \norm{\inv a}_\mA^3 \norm{a}_{\besov \infty r \mA}^2 \\
&\leq C'_r \norm{\inv a}_\mA^3 \norm{a}_{\bessel  r \mA}^2 \,.
\end{align*}
In the last inequalities we have used  the  norm control~\eqref{eq:ncbesov} for Besov spaces and the inclusion relation \eqref{eq:beinbe} between Bessel and Besov spaces.
The second term can be estimated as
\begin{align*}
  \bignorm{\int_{\epsilon \leq \abs t  \leq 1}\frac{\inv a \Delta_t(a) \inv a}{\abs t  ^r} \muleb{t}}_\mA 
 =& \bignorm{\inv a \Bigl(\int_{\epsilon \leq \abs t  \leq 1}\frac{ \Delta_t(a) }{\abs t  ^r} \muleb{t}\Bigr) \inv a}_\mA \\
\leq & C_r\norm{\inv a}^2_\mA \norm{a}_{\bessel r \mA}.
\end{align*}
As $\norm{a}_{\bessel r \mA} \asymp \norm{a}+ \sup _{0< \epsilon <1}  \bignorm{\int_{\epsilon \leq \abs t  \leq 1}  \frac{\Delta_t(\inv a)}{\abs{t}^{r}}\muleb{t}}_\mA$,
and $\norm{\inv a}_\mA \norm{a}_{\bessel r \mA} \geq 1$
 we conclude that, for $r<1$,
\begin{equation}
  \label{eq:49}
  \norm{\inv a}_{\bessel r \mA} \leq C_r \norm{\inv a}_\mA^3
  \norm{a}_{\bessel r \mA}^2 \, .
\end{equation}
If $r>1$, an inductive argument as in the proof of
Proposition~\ref{prop:besov-nci-basic}~ shows that, for
$\norm{a}_{\bessel r \mA} \leq 1$ and $\norm {\inv a}_\mA \leq
1/\delta$ one obtains 
\begin{equation*} 
  \norm{\inv a}_{\bessel r \mA} \leq C_r \delta^{-3^r} \, . \qedhere
\end{equation*}
\end{proof}
\begin{rems}
(i)  Theorem~\ref{prop:thm-besspot_baic}
remains true for the action of  a $d$-parameter automorphism group,
but requires more notations. 

(ii)   We expect that the correct  asymptotic behavior in the Bessel
algebra $\mP _r(\mA )$  is of the form  
\begin{equation*}
  \norm{\inv a}_{\bessel r \mA} \leq C_r \delta^{-r-2} \,.
\end{equation*}
Possibly this could be proved by using equivalent norms for  $\bessel
r \mA$ for $r>2$  that involve  a hyper-singular integral  with higher
differences as  in \cite{Wheeden68}, but we do not pursue this detail
further.  
\end{rems}
\subsection{Norm Controlled Inversion in Dales-Davie Algebras}
\label{sec:norm-contr-invers}
As a final version of smoothness in the presence of a derivation
$\dd$ we now look at the Dales-Davie algebras.  The
Dales-Davie algebras   are
determined by growth conditions on the sequence
$\norm{\dd^k(a)}_\mA , k \in \bn_0$ and are related  to the
Gelfand-Shilov spaces of test functions in analysis. For scalar functions
they were introduced in~\cite{dales73}. 
\begin{defn}
Let $M =(M_k)_{k\geq0}\subseteq \br $  be  a sequence of positive
numbers satisfying  $M_0=1$ and 
$ 
\frac{M_{k+l}}{(k+l)!} \geq  \frac{M_k}{k!} \frac{M_l}{l!}$ for all  $k,l \in \bn_0$.
The \emph{Dales-Davie algebra} $\dadas 1 M \mA$ consists of the elements $a \in \mA$ with finite norm
\[
\norm{a}_{\dadas 1 M \mA} 
=\sum_{k=0}^\infty \inv{M_k}{\norm{\dd^k(a)}_\mA} \,.
\]
\end{defn}
Then  $\dadas 1 M \mA$ is  a \BA , and by the results
in~\cite{klotz12pre} $\dadas 1 M \mA$ is \IC\ in \mA . See also
\cite{honary07} for a proof in the commutative setting. 
Once again, this form of smoothness admits \nci . 
\begin{thm}\label{thm-dd-nci}
 Let $\mA $ be a \BA\ with a closed unbounded derivation $\dd
 $. Set  
\begin{equation}\label{eq:icsuff}
A_m=\Bigl (\sup \bigset{ \frac{k!}{M_k}\prod_{j=1}^m
  \frac{M_{l_j}}{l_j!} \colon l_j \geq 1 \text{ for } 1 \leq j \leq m,
  \sum_{j=1}^m l_j =k} \Bigr)^{1/m} \, .
\end{equation}
If  $\lim_{m \to \infty}A_m=0$,
then $\dadas 1 M \mA$ admits  \nci\ in \mA  .
\end{thm}
Note that the definition implies that $A_m \leq 1$ for all $m \in \bn$
and that $A_m$ is a decreasing sequence.
\begin{proof}
Once again we  use the iterated quotient rule of Lemma~\ref{invnormderiv}  and
proceed as in the  proof of Proposition~\ref{ncideralg}. Details are
therefore  left to the reader.
  \begin{align*}
   \norm{\inv a}_{\dadas 1 M \mA} &= \sum_{k=0}^\infty \frac{\norm{\dd^k(\inv a )}_\mA}{M_k}\\
    & \leq \cdots \leq  \norm{\inv a}_\mA +   \sum_{k=1}^\infty \frac{k!}{M_k}
    \sum_{m=1}^k\norm{\inv a}_\mA^{m+1} \sum_{\substack{k_1+\dotsc
        k_m=k\\ k_j\geq1}} 
    \prod_{j=1}^m\frac{\norm{\dd^{k_j}(a)}_\mA}{k_j!}\\
    &   \leq  \norm{\inv a}_\mA +  \sum_{m=1}^\infty\norm{\inv a}_\mA^{m+1} \sum_{k=m}^\infty \,   \sum_{\substack{k_1+\dotsc k_m=k\\ k_j\geq1}} \frac{k!}{M_k} \prod_{j=1}^m\frac{M_{k_j}}{k_j!}\prod_{j=1}^m\frac{\norm{\dd^{k_j}(a)}_\mA}{M_j}\\
&   \leq  \norm{\inv a}_\mA +  \sum_{m=1}^\infty\norm{\inv a}_\mA^{m+1} \sum_{k=m}^\infty \,  \sum_{\substack{k_1+\dotsc k_m=k\\ k_j\geq1}} A_m^m\prod_{j=1}^m\frac{\norm{\dd^{k_j}(a)}_\mA}{M_j} \\
    & \leq \cdots \leq  \norm{\inv a}_\mA+ \sum_{m=1}^\infty\norm{\inv a}_\mA^{m+1} A_m^m (\norm{a}_{\dadas 1 M \mA} - \norm{a})^m
  \end{align*}
If  $\norm{\inv a}_\mA \leq 1/\delta$ and $\norm{a}_{\dada  \mA} \leq 1$, the last inequality simplifies to
\begin{equation}\label{eq:dadaest}
  \norm{\inv a}_{\dada   \mA}  \leq \inv\delta + \sum_{m=1}^\infty
  \delta^{-m-1} A_m^m \, .
\end{equation}
We now  choose an index $m_\delta\in \bn$,    such that $A_m < \delta/2$ for all $m>m_\delta$. This implies that
\begin{align*}
   \norm{\inv a}_{\dada   \mA}  &\leq \inv\delta + \sum_{m=1}^{m_\delta} \delta^{-m-1} A_m^m +  \sum_{m=m_\delta+1}^\infty \delta^{-m-1}\bigl(\frac {\delta}{2}\bigr)^m\\
    &    \leq \inv\delta + \sum_{m=1}^{m_\delta} \delta^{-m-1} A_m^m +  \sum_{m=m_\delta+1}^\infty \delta^{-m-1}\bigl(\frac {\delta}{2}\bigr)^m\\
    &    \leq \inv\delta + \sum_{m=1}^{m_\delta} \delta^{-m-1} A_m^m +  \inv\delta\sum_{m=m_\delta+1}^\infty {2}^{-m} \leq 2 \inv\delta + \sum_{m=1}^{m_\delta} \delta^{-m-1} A_m^m\\
    &    \leq 2 \inv\delta + \inv\delta
    \frac{\delta^{-m_\delta}-1}{\inv \delta -1} \leq 4
    \delta^{-m_\delta} \text{   for   } \delta \leq 1/2 \, , 
\end{align*}
as  $A_m \leq 1$.
Note that $m_\delta $ is independent of $\norm{\inv a}_\mA$ and
$\norm{a}_{\dada \mA}$ and depends only on the sequence $A_m$ and thus
 on  $M_k/k!$.
\end{proof}

Let us work out several important special cases. 

(i) If $M_k = 1/k!$ for $k\leq K$ and $M_k=0$ for $k>K$, then  $\dadas
1 M \mA = \mD (\dd^K)$. We may choose $m_\delta = K+1$ independent of
$\delta $ and recover the statement of Proposition~\ref{ncideralg}.

(ii) If $M_k = 1/k!$ for all $k\in \bn$, then $\dadas
1 M \mA $ is an algebra of analytic elements in $\mA $. In this case
$A_m=1$ and Theorem~\ref{thm-dd-nci} is not applicable. 
 This is not surprising, because we already know that this 
subalgebra is not \IC\ in $\mA $~\cite{klotz12pre}.

(iii) If $M_k = 1/k!^r$ for  $k\in \bn$ and for some $r>1$, then 
 a straightforward calculation shows that
\begin{equation}\label{eq:gevrey}
  A_m =\sup_{l_j \geq 1} \Bigl( \frac{l_1 ! \dotsm l_m!}{(l_1  +\dotsb
    +l_m)!}  \Bigr)^{\frac{r-1}{m}} =m!^{\frac{1-r}{m}} \, ,
\end{equation}
which tends to $0$ for $m \to \infty$ and $r>1$. In \cite{klotz12pre}
it is shown that $\dadas 1 {k!^r} \mA$ is \IC\ in \mA\ if  $r>1$. In
this case the estimate \eqref{eq:dadaest} reads as 
\begin{equation} \label{eq:asympexp}
  \norm{\inv a}_{\dadas 1  {k!^r} \mA} \leq \inv \delta + \sum_{m=1}^\infty \frac{\delta^{-m-1}}{m!^{r-1}}= \inv{\delta} v_{r-1}(\inv{\delta})
\end{equation}
where we have introduced  the weight function 
\[
v_r(x)=\sum_{l=0}^\infty \frac{\abs x ^l}{l!^r} \, , \qquad x\in \br  \,.
\]

\section{Norm Controlled Inversion in Matrix Algebras with Off- Diagonal Decay}
\label{sec:norm-contr-invers-3}

The goal of this section is twofold. First  we will show that the asymptotic
estimates obtained for smooth subalgebras are almost sharp. To do so,
we test on a particular class of Banach algebras consisting of
matrices with off-diagonal decay.  Then we will investigate  \nci\ in
subalgebras of $\bop $. While the abstract methods of Section~2 can be
exploited to construct subalgebras of $\bop $ with \nci , the standard
matrix algebras that  describe  off-diagonal decay conditions do
not fall under this theory. In this case the arguments are much more
sophisticated and  essentially due to Baskakov~\cite{Baskakov97}.


\subsection{Matrix Algebras}
\label{sec:matrix-algebras}

First we define the standard  classes of matrices with \odd. For  more
comprehensive treatments see,
e.g.,~\cite{Baskakov97,Bas97,GL04a,Jaffard90,Sun05}. In the following
$A$ is always a matrix over the index set $\bz $ with entries $A(k,l),
k,l \in \bz $.

The \emph{Jaffard algebra} $\mJ_r$, $ r>1$, is defined by the norm
\begin{equation}
  \label{eq:jaffnorm}
  \norm{A}_{\mJ_r}=\sup_{k,l \in \bz}\abs{A(k,l)}(1+\abs{k-l})^r.
\end{equation}
Explicitly, $A \in \mJ_r \Leftrightarrow |A(k,l)| \leq C (1+|k-l|)^{-r} $, so the norm of $\mJ_r$ describes
polynomial decay off the diagonal. 

Assume that the weight $v$ on \bz\ is submultiplicative and  satisfies
the GRS condition, i.e., $v$ satisfies $v(0)=1$, $v(k+l) \leq
v(k) v(l)$, and $\lim_{\abs k \to
  \infty}v(k)^{1/k}=1$. 
Then the \emph{algebra of convolution-dominated matrices} $\mC_v$,  (sometimes called the
Baskakov-Gohberg-Sj\"ostrand algebra) consists of all matrices $A$, such that the norm
\begin{equation}
  \label{eq:basknorm}    
  \norm{A}_{\mC_v}= \sum_{k \in \bz}\sup_{l \in \bz} \abs{A(l,l-k)} v(k)
\end{equation}
is finite. For the polynomial weights  $v(k)=(1+\abs k)^r,\, r>0$ we
use the  abbreviation $\mC_r$. 


The crucial observation to apply the abstract  theory of Section~2 is
the identification of  weighted \MA s  with  smoothness spaces. Precisely,  the
automorphism group is 
\begin{equation*}
  \psi_t(A)=M_t A M_{-t}, \quad  M_t x (k) = \cexp [ k t ] x(k) \, ,
\end{equation*}
and the corresponding derivation is defined entrywise as $\dd(A) (k,l) =
(k-l) A(k,l)$. 
Then  the following identifications  hold~\cite{grkl10,klotz12, klotz12pre}.

\begin{align}
  \label{eq:jaffislipschitz}
  \mJ_r &= \besov \infty {r-1-\epsilon} {\mJ_{1+\epsilon}}, \quad r-1-\epsilon >0, \epsilon >0 \,,   \\
  \label{eq:baskisbesov}
  \mC_r &= \besov 1 r {\mC_0}, \quad r>0 \,,\\ 
  \label{eq:baskisdd}
  \mC_{v_M} &= \dada {\mC_0} ,\quad \text{where  } v_M(k)=\sum_{l=0}^\infty {\abs k ^l}{\inv M_l}  \,. 
\end{align}

These results combined with Theorems~\ref{thm-besov-nci} and
~\ref{thm-dd-nci}  show that  (i)  $\mC_r$ and  $\mC_{v_M}$ admit \nci\ in
$\mC_0$ with a precise expression for the controlling function, and
that likewise  (ii) $\mJ_r $ admits \nci\ in $\mJ _{1+\epsilon }$.

\subsection{Optimality of the Asymptotic  Estimates}
\label{sec:sharpness-estimates}
We use the matrix algebras defined above to test whether the
estimates of the control function $h$ obtained in
Section~\ref{sec:quotient} are asymptotically sharp. This will be done
by comparing the norm of the inverse of an operator in $\mC_r$ to the
norm in $\mC_0$. (A similar test could have been carried out in
$\mJ_{s+r}$ and $\mJ_s$, but it is quite tedious to obtain the
estimates.)

To avoid distraction, we note right away that 
\begin{equation}
  \label{eq:oslo5}
  e^\gamma \gamma ^{-r-1} \Gamma (r+1) \leq \sum _{k=0}^\infty (k+1)^r
  e^{-\gamma k} \leq 2 e^\gamma \gamma ^{-r-1} \Gamma (r+1) \, .
\end{equation}
This follows immediately from the integral test. The infinite series
can be expressed by a polylogarithm.

  For $0 < \gamma <1 $ we define the operator $C_\gamma
  =\id-e^{-\gamma}T_1$ on $\ell^2(\bz)$, where the translation
  operator is $T_k x(l)=x(l-k)$. Then  $C_\gamma$ is a Toeplitz matrix with
  ones on the main diagonal and the value $e^{-\gamma}$ on the
  first side diagonal. The inverse of $C_\gamma$ is given by the
  geometric series 
\[
C_\gamma^{-1}= \sum_{k=0}^\infty e^{-\gamma k} T_k \,.
\]
With the help of \eqref{eq:oslo5} the norms of $C_\gamma$ and $\inv C_\gamma$ in $\mC_r$ are given by
\begin{align}
 & \norm{C_\gamma}_{\mC_r}=1+2^r e^{-\gamma}\,, \notag  \\
 & \norm{\inv C_\gamma}_{\mC_r}=\sum_{k=0}^\infty(1+k)^r e^{-\gamma k}
 \asymp e^\gamma \Gamma(r+1) \gamma^{-r-1} \,. \label{eq:oslo1}
\end{align}

 We set 
\[
\widetilde C_\gamma= \frac{1}{1+2^r e^{-\gamma}} C_\gamma \,,
\]
so that  $\norm{ \widetilde C_\gamma}_{\mC_r}=1$. Comparing
\eqref{eq:oslo1} for $r=0$ and general $r>0$, we obtain that 
\begin{equation*}
  \norm{\inv {\widetilde  C}_\gamma}_{\mC_0} \asymp (1+  e^{-\gamma}) e^\gamma \inv \gamma = \bigo(\inv \gamma) \quad \text{for} \quad \gamma \to 0
\end{equation*}
whereas 
\begin{equation*}
  \norm{\inv {\widetilde  C}_\gamma}_{\mC_r}  = \bigo( \gamma^{-r-1})
  \quad \text{for} \quad \gamma \to 0 \, .
\end{equation*}
 Setting $1/\delta= \norm{\inv {\widetilde  C}_\gamma}_{\mC_0}$, we
 obtain 
\begin{equation}
  \label{eq:ncibaskconv}
  \norm{\inv {\widetilde  C}_\gamma}_{\mC_r}=\bigo(\delta^{-r-1}) \quad \text{for} \quad \delta \to 0 \,.
\end{equation}
This should be compared with Theorem~\ref{thm-besov-nci} which yields
the asymptotic rate $\mO (\delta ^{- \floor r - 2})$.
For integer values $r=k\in \bn $ we have $
\mC_k = \mD(\dd^k )(\mC_0)$,  and  Proposition~\ref{ncideralg} shows
that $\mO (\delta ^{-k-1})$ is the precise asymptotic rate for the
norm control of $\mC _k$ in $\mC _0$. 



Next we discuss \nci\ for weights in the Dales-Davie algebras for
matrices and weights with  super-polynomial growth. We restrict our attention to
the  special kind of subexponential weights
\begin{equation}
  \label{eq:ansubexpweight}
 v_{r}(x) = \sum_{k=0}^\infty \frac{\abs{x}^k}{k!^r}, \quad  r>0 \, , 
\end{equation}
that have  been introduced  earlier. 
As stated in (\ref{eq:baskisdd}) above, $\mC_{v_{r}} = \dadas 1
{k!^r} {\mC_0}$. Therefore we already  know  that for a matrix $A\in \mC _0$
satisfying  $\norm{A}_{\mC _{\nu _r}} \leq 1$ and $\norm{ \inv A}_{\mC
  _0} \leq 1/\delta$, the inverse matrix  satisfies 
\begin{equation} \label{eq:asympexpa}
  \norm{\inv A}_{\mC_{\nu _r}} \leq \inv{\delta} v_{r-1}(\inv{\delta})
  \, .
\end{equation}

For comparison we  determine the growth of the weight $v_{r}$ more exactly.
\begin{lem}
  The function $\phi_r(z) = \sum_{k=0}^\infty \frac{z^k}{k!^r}$ is an
  entire function of order $1/r$ and of type $r$. As a consequence,
  for every $\epsilon>0$ there exists an $x_\epsilon$, such that 
\begin{equation*}
   v_r(x) = \phi _r(|x|) <  e^{(r+\epsilon) |x|^{1/r}}
\end{equation*}
for  $|x| >x_\epsilon$.
\end{lem}
\begin{proof}
Recall  that for an entire function  
$f(x)=\sum_{k=0}^\infty a_kx^k$, the order $\rho _f$ and  the type
$\sigma_f$ are defined as 
 $ \rho_f =\varlimsup_{k \to \infty} \frac{k \log k}{\log
   (1/{\abs{a_k})}} $,  and    $\sigma_f = \inv{ (\rho_f \,e )} 
 \varlimsup_{k \to \infty} k \abs {a_k}^{\rho_f/k}$ respectively, see,
 e.g., ~\cite{Levin96}. 
A  straightforward calculation with  Stirling's formula for $k!$  yields
   \begin{align*}
     \rho_{\phi_r}&=\varlimsup_{k \to \infty} \frac{k \log k}{\log k!^r}
                 =\varlimsup_{k \to \infty} \frac{k \log k}{r k \log
                   k} = 1/r \,, \\
     \sigma_{\phi_r} &= \frac{r}{e} \varlimsup_{k \to \infty} k
     ({k!}^{-1/k})  = r \,.\qedhere
   \end{align*} 
 \end{proof}

In order to test the sharpness of the estimate \eqref{eq:asympexpa},  we
use again the Toeplitz  matrix $C_\gamma =\id-e^{-\gamma}T_1 $.  The
identification~\eqref{eq:baskisdd} of $\mC _{v _r}$ with a
Dales-Davie algebra leads to the following expression for the norm of
$C_\gamma $ in $\mC _{v _r}$
\begin{align*}
  \norm{C_\gamma^{-1}}_{\mC_{v_{r}}}= \norm{C_\gamma^{-1}}_{\dadas 1 {(k!^r)} {\mC_0}}
                              =\sum_{k=0}^\infty \frac{\norm{\dd^k (\inv{C_\gamma})}_{\mC_0}}{k!^r}  
\end{align*}
As $\dd (A)(k,l) = (k-l) A(k,l)$,  the norm in each term is 
\begin{align*}
  \norm{\dd^k (\inv{C_\gamma})}_{\mC_0} = 1 + \sum_{j=1}^\infty j^k
  e^{-\gamma j}  \asymp e^\gamma  \gamma^{-k-1 }
\Gamma(k+1) \, ,  
\end{align*}
again with \eqref{eq:oslo5}. So  we obtain
\begin{align*}
  \norm{C_\gamma^{-1}}_{\mC_{v_r}}  \asymp \inv \gamma
  \sum_{k=0}^\infty \frac{(1/ \gamma)^k}{k!^{r-1}} = 
\frac{1}{\gamma}   v_{r-1}\big(\frac{1}{ \gamma }\big) \, .
\end{align*}
The comparison with~\eqref{eq:asympexp} shows that 
 the estimate (\ref{eq:asympexpa}) is asymptotically sharp. 
\subsection{Norm Control of Matrix Algebras in \bop}
\label{sec: ncibop}
\Nci\ of the Jaffard algebra $\mJ_r$ and the algebra of convolution dominated matrices of polynomial decay $\mC_r$ in \bop\ has been first investigated in a fundamental paper of Baskakov~\cite{Baskakov97}. 
To formulate his main result, we need  some notation.

Assume that $A \in \mC_0$ and let  $d_A(k) = \sup
_{l\in \bz } |A(l,l-k)|$ be the supremum of the $k$-th side diagonal
of $A$.  The  error of approximating $A$ by
$k$-banded matrices in $\mC_0$ is then  $E_k(A)=\sum_{\abs m \geq k+1}
d_A(m)$.  
We also need the quantities
\begin{equation}\label{eq:lr}
  \ell _r(A) = 8 \norm{\inv A}_\bop \sum_{k=0}^\infty \Bigl(1-\frac{1}{24 \kappa(A) +1}\Bigr)^k (1+k)^r  \,,\quad r>0 \,,
\end{equation}
and
\begin{equation}
  \label{eq:lrj} 
  \widetilde \ell _r(A) =  \gamma _r \norm{\inv A}_\bop\sup_{k \in
    \bn_0}  \Bigl(1-\frac{1}{24 \kappa(A) +1}\Bigr)^k (1+k)^r \,,
  \quad r>1  \, . \footnote{In ~\cite{Baskakov97} the constant $24 = 4\cdot 3\cdot 2$
  seems to be misspelled as $8\cdot 32$. In the following the precise
  value is not important.}
\end{equation}
The constant $\gamma _r$ comes from the from the convolution
inequality $v_r ^{-1} \ast v_r ^{-1} \leq \gamma _r v_r ^{-1}$ for the
polynomial weight function $v_r(k) = (1+|k|)^r$ and can  be estimated
by $\gamma _r \leq 2^{r+1}\tfrac{r+1}{r-1}$.  

Furthermore, we need  the function
\begin{equation} \label{eq: Phi}
  \Phi_{A,r}(t) = \min \set{k \in \bn  \colon  \max \Big(2 \cdot 3^r k^{-r}
    \norm{A}_{\mC_r} \ell _r(A), 2 E_k(A) \norm{\inv A}_\bop \Big) \leq t}
  \, . 
\end{equation}
Then a simplified  version of Baskakov's result for the matrix
algebras $\mJ _r$ and $\mC _r$   reads as
follows.~\footnote{The most general version contains   more parameters that we  have already eliminated.} 
\begin{thm}[{\cite[Thm.~6]{Baskakov97}}] \label{baskthm}
  Assume that $A$ is invertible in \bop.
  \begin{enumerate}
  \item[(i)] If $A \in \mC_r$ and  $r>0$, then $\inv A \in \mC_r$, and
    \begin{equation*}
      \norm{\inv A}_{\mC_r} \leq 4  \ell _r(A) \inf_{0\leq t\leq1/2}
      \Phi_{A,r}(t)(1+\Phi_{A,r}(t))^r \, .
    \end{equation*}
  \item[(ii)] If $A \in \mJ_r$ and $r>1$,  then $\inv A \in \mJ_r$, and
    \begin{equation*}
  \norm{\inv A}_{\mJ_r} \leq 4 \widetilde \ell _r(A) \Big( 2 + \big(
  2\cdot 3^r \norm{A}_{\mJ _r} \widetilde \ell _r(A)
  \big)^{\frac{1}{r-1}} \Big) ^r 
       \qquad
       {(*)} \footnote{ In ~\cite{Baskakov97} the formula $(*)$ is stated
  incorrectly  as
\[
\norm{\inv A}_{\mJ_r} \leq 4 \widetilde \ell _r(A) \Big( 2 + \big(
  2\cdot 3^r \norm{A}_{\mJ _r} \widetilde \ell _r(A)
  \big)^{\frac{1}{r-1}} \Big)\, . 
\]}
    \end{equation*}
  \end{enumerate}
\end{thm}
Since the definition of $\ell _r(A) $ and $\Phi _{A,r}$ involves  the
norms $\norm{A}_{\mC _r}$ and $\norm{\inv A}_{\bop } $ and the
approximation error $E_k(A)$, 
Theorem~\ref{baskthm} can be interpreted as a form of norm control for
 $\mC _r$ and $\mJ _r$ in $\bop $. 
Baskakov's theorem is rather deep, but unfortunately the result as
stated provides little information about the nature of the control
function. In the following we derive an explicit expression for the
control function by  using some approximation  properties of the matrix algebras
$\mC _r$ and $\mJ _r$.  

\begin{prop} \label{Brueckenlemma}
 Assume that $A$ is invertible in \bop.
\begin{enumerate}
\item[(i)] If $A \in \mC_r$  and $r>0$,  then $\inv{A} \in \mC _r$ and 
  \begin{align*}
     \norm{\inv A}_{\mC_r} &\leq C_r \norm{A}_{\mC_r}^{1+\frac{1}{r}}\norm{A}_{\bop}^{2 r+\frac{2}{r}+3} \norm{\inv A}_{\bop}^{2 r + \frac{2}{r} +5} \\ 
     & \leq  C_r \norm{A}_{\mC_r}^{2 r+\frac{3}{r} +4} \norm{\inv
       A}_{\bop}^{2 r + \frac{2}{r} +5} \, .
  \end{align*}
In particular, if  $\norm{A}_{\mC_r}\leq 1$, and $\norm {\inv A}_{\bop} \leq 1/\delta$, then  
\[
\norm{\inv A}_{\mC_r} \leq C_r \delta^{-2r-\frac{2}{r}-5} \quad \text{
as } \quad \delta \to 0 \,.
\]
\item[(ii)] If $A \in \mJ_r$ and  $r>1$,  then $\inv{A} \in \mJ _r$
  and 
  \begin{align*}
    \norm{\inv A}_{\mJ_r} &\leq \widetilde{C_r}
    \norm{A}_{\mJ_r}^{1+\frac{1}{r-1}}\norm{A}_{\bop}^{
      2r+1+\frac{1}{r-1}} \norm{\inv A}_{\bop}^{ 2r + 3 + \frac{2}{r-1}} \\
 & \leq \widetilde{C_r} \norm{A}_{\mJ_r}^{2r+2 +\frac{2}{r-1}} \norm{\inv
   A}_{\bop}^{ 2 r+3 + \frac{2}{r-1}}  \, .
  \end{align*}
 In particular, if  $\norm{A}_{\mJ_r}\leq 1$, and $\norm {\inv A}_{\bop} \leq 1/\delta$, 
\[
\norm{\inv A}_{\mJ_r} \leq \widetilde{C_r} \delta^{-2r- \frac{2}{r-1}-3} \quad \text{for} \quad \delta \to 0 \,.
\]
\end{enumerate}
\end{prop}
\begin{proof}
We divide the proof of (i) into several steps. We note right away that
for the  polynomial weights $(1+|k|)^r$ the function $\Phi _{A,r}(t)$
is decreasing, therefore we may replace  $\inf _{0\leq t\leq 1/2} \Phi
_{A,r}(t)$ by $\Phi _{A,r}(1/2)$. We will enumerate the  intermediate
constants consecutively by $c_1,c_2, \dots $ and emphasize that these
depend only on $r$, but not on $A$.

\emph{Step 1:}
We first estimate $\ell _r(A)$. 
Set $\beta=\frac{1}{24 \kappa(A) +1}$. Then $\beta< 1/2$ and
$\inv{(1-\beta )} \leq 2$. Using~\eqref{eq:oslo5} with $\gamma = \log \inv{ (1-\beta )}$, we obtain
\begin{align}
   \ell _r(A) &= 8\norm{\inv A}_{\bop} \sum_{k=0}^\infty (1-\beta)^k
   (1+k)^r  \notag \\
&= 8 \norm{\inv A}_{\bop} \sum_{k=0}^\infty  (1+k)^r e^{-k \log (1-\beta
  )^{-1} } \notag  \\ 
&\leq  16  \norm{\inv A}_{\bop}\frac{ \Gamma(r+1)}{(1-\beta) (\log \inv{(1-
  \beta)})^{r+1}} \,. \label{eq:oslo2} 
\end{align}
Since
\begin{equation*}
  \frac{1}{\log \inv{(1-\beta)}} \leq 1/\beta = 24 \kappa(A)+1 \leq 25 
  \kappa (A) 
\end{equation*}
by the Taylor expansion of $\log(1-x)$, \eqref{eq:oslo2} implies that
\begin{equation}\label{eq:mrest}
  \ell _r(A)  \leq   c_1 \norm{A}_\bop^{r+1} \norm{\inv A}_\bop^{r+2} \,.
\end{equation}
\emph{Step 2:}
We use \eqref{eq:mrest} to  estimate  $\Phi_{A,r}(1/2)$. Obviously,
\begin{equation*}
  c_2 k^{-r} \norm{A}_{\mC_r} \ell _r(A) \leq \tfrac{1}{2}
\end{equation*}
is satisfied, if
\begin{equation*}
  c_2 k^{-r} \norm{A}_{\mC_r}\norm{A}_\bop^{r+1} \norm{\inv
   A}_\bop^{r+2} \leq \tfrac{1}{2} \, .
\end{equation*}
For this we need 
\begin{equation}
  \label{eq:vie1}
k\geq  c_3   \norm{A}_{\mC_r}^{1/r}\norm{A}_\bop^{1+1/r}
\norm{\inv A}_\bop^{1+2/r}\, .  
\end{equation}
The  second term in  (\ref{eq: Phi}) is estimated by 
\begin{align*}
  E_k(A)=&\sum_{\abs m \geq k+1}  d_A(m)  = \sum_{\abs m \geq k+1}  d_A(m) (1+ \abs m)^r (1+ \abs m)^{-r} \\
 \leq&  (1+  k)^{-r} \norm{A}_{\mC_r} \leq k^{-r} \norm{A}_{\mC_r} \,.
\end{align*}
So the inequality
\begin{equation*}
  2 E_k(A) \norm{\inv A}_\bop \leq \tfrac{1}{2}
\end{equation*}
is satisfied for 
\begin{equation}\label{eq:k2}
 k \geq  4^{1/r} \norm{A}_{\mC_r}^{1/r} \norm{\inv A}_\bop^{1/r}  \,.
\end{equation}
As \eqref{eq:vie1} is stronger than \eqref{eq:k2}, possibly after adjusting
the constants,  we obtain
\begin{equation}
  \label{eq:Fiest}
  \Phi_{A,r}(1/2) \leq c_4  \norm{A}_{\mC_r}^{1/r} \norm{A}_\bop^{1+1/r} \norm{\inv A}_\bop^{1+2/r} \,.
\end{equation}
\emph{Step 3:} We combine  the estimates for $\ell _r(A)$ and $\Phi_{A,r}$
and substitute in  condition (i) of Theorem~\ref{baskthm}. 
In conclusion, we obtain 
\begin{align*}
\norm{\inv A}_{\mC_r} 
&\leq  4 \ell _r(A) \Phi_{A,r}(\tfrac{1}{2})(1+ \Phi_{A,r}(\tfrac{1}{2}))^r \\
&\leq c_5  \ell _r(A)( \Phi_{A,r}(1/2))^{r+1} \\
&\leq c_6 \norm{A}_\bop^{r+1} \norm{\inv A}_\bop^{r+2}
\Bigl(\norm{A}_{\mC_r}^{1/r} \norm{A}_\bop^{1+1/r} \norm{\inv
  A}_\bop^{1+2/r} \Bigr)^{r+1} \\ 
&= C_r \norm{A}_{\mC_r}^{\frac{r+1}{r}}\norm{A}_\bop^{2r +3 +
  1/r}\norm{\inv A}_\bop^{2r +5 +2/r } \,.
\end{align*}
The estimates for $\delta$ follow now immediately.
\\

The estimate for $\norm{\inv A}_{\mJ_r}$ follows in a  similar, but
simpler 
fashion from Theorem~\ref{baskthm}(ii).  We now need to estimate
$\widetilde \ell _r(A)$. 
Using again
$\beta=\frac{1}{24 \kappa(A) +1}< 1/2$ and arguing as in
\eqref{eq:mrest},  we obtain the  estimate for $\widetilde \ell _r(A)$
as 
\begin{align*}
  \widetilde \ell _r(A) & = \gamma _r \norm{\inv A}_{\bop} \sup_{k \geq 0}(1-\beta)^k(1+k)^r \\
=& \gamma _r \norm{\inv A}_{\bop}\frac{ e^{-r} r^r}{(1-\beta) (\log (1-\beta)^{-1})^r} \\
& \leq c_7 \norm{ A}_{\bop}^r\norm{\inv A}_{\bop}^{r+1} \, .
\end{align*}
Then according to Theorem~\ref{baskthm}(ii), 
\begin{align*}
  \norm{\inv A}_{\mJ_r} &\leq 4 \widetilde \ell _r(A) \Big( 2 + \big(
  2\cdot 3^r \norm{A}_{\mJ _r} \widetilde \ell _r(A)
  \big)^{\frac{1}{r-1}} \Big) ^r \\  
&\leq c_8 \widetilde \ell _r(A) (\norm{A}_{\mJ_r} \widetilde \ell _r(A))^{\frac{r}{r-1}}\\
&\leq c_9 \norm{ A}_{\bop}^r\norm{\inv A}_{\bop}^{r+1} \Bigl(\norm{A}_{\mJ_r} \norm{ A}_{\bop}^r\norm{\inv A}_{\bop}^{r+1}\Bigr)^{\frac{r}{r-1}}\\
&= \widetilde{C_r} \norm{A}_{\mJ_r}^{\frac{r}{r-1}}\norm{ A}_{\bop}^{2r+ 1
  +{\frac{1}{r-1}}}\norm{\inv A}_{\bop}^{2r+3+\frac{2}{r-1}} \, . 
\end{align*}
The remaining assertions are immediate.
\end{proof}

\begin{rems}
 (i) 
  With some painful bookkeeping of the constants in each step one may
  give a numerical value for the constants in Proposition~\ref{Brueckenlemma}. For
  $C_r$ we have obtained the value $C_r = 128 \cdot 50^r \cdot 64^{1/r}
  \Gamma (r+1)^{1+1/r} $. The main insight of the explicit norm
control is how the off-diagonal decay depends on the condition number
of $A$. However, the growth  of $C_r$ in $r$  shows that the off-diagonal  decay is effective
only in an asymptotic range.

(ii) The unraveling of an explicit norm control function  is relatively
straightforward  when compared to the depth of Baskakov's Theorem, but
we believe that the  explicit version in
Proposition~\ref{Brueckenlemma} is much more accessible and will be useful in  
some  applications. 

(iii)  As in all statements about \nci , the quality of the norm
control degrades as the smoothness parameter $r$ tends to
infinity.  In addition, if $r\to 0$, then \nci\ for $\mC _r$ also
deteriorates. This is compatible with the fact that $\mC _0$ 
does not admit \nci\ in \bop , because $\mC_0$ contains the algebra of absolutely convergent Fourier series as a
commutative closed subalgebra. 

(iv) For the commutative algebra of Toeplitz matrices in $\mC _r $
and $\mJ _r$  the following estimates can be found in
\cite[Sec. 3.8.4 (ii)]{elfalla98}. If $A$ is a Toeplitz matrix in $\mC _r$,
$\norm{A}_{\mC _r} \leq 1$ and $ \norm{\inv A}_\bop \leq
  1/\delta$, then
  \begin{equation*}
    \norm{\inv A}_{\mC _r} \leq 
     \begin{cases}
      C \delta^{-4 r -2} & \quad \text{ for } r>1, \\
C \delta^{-( r+1)^2(4 r +1)/r^2 -2} & \quad \text{ for } r>0 \, .
    \end{cases}
 \end{equation*}
This is weaker than Proposition~\ref{Brueckenlemma}, but on the other
hand, the techniques of~\cite{elfalla98} work for  much more general
commutative Banach algebras. 
\end{rems}
\subsection{Super-algebraically growing weights.}
\label{sec:super-grow-weights}

Baskakov's approach   does not work for  weights of
super-polynomial growth. (This is not obvious from   the general
assumptions of ~\cite{Baskakov97}, but becomes clear in his Lemma~3, where a 
``doubling condition'' is required. If a submultiplicative weight $v$
satisfies the condition $\sup _{k\in \bz }  v(km)/v(k) < \infty $ for
$m\geq 2$, then $v$ grows at most polynomially.)   To obtain \nci\ for
the  the  case of weights with 
super-polynomial   growth,  we  combine the results on 
Dales-Davie algebras with the transitivity of \nci , and we obtain the
\nci\ for convolution dominated matrices with subexponential decay off
the diagonal. Although the following statement is  probably far from
optimal, it is the first result of this type for non-commutative
matrix algebras. 

\begin{prop}
  If $\norm{A}_{\mC_{v_{r}}}\leq 1$ and $\norm {\inv A }_{\bop} \leq 1/\delta$ then
\[
\norm{\inv A}_{\mC_{v_r}} \leq C_0 {\delta^{-9}} v_{r-1}({ C_1 \delta^{-9}})
\leq C \delta ^{-9} \exp \Big( C_1' (r-1+\epsilon) \delta
^{-\frac{9}{r-1}}\Big)  \,  .
\]
\end{prop}

\begin{proof}
Clearly, $\mC _{\nu _r} \subseteq \mC _s$ for all $s\geq 0$ and
$\norm{A}_{\mC _s} \leq \norm{A}_{\mC _{\nu _r}} = 1$,   therefore
Proposition~\ref{Brueckenlemma}(i) implies that 
$$
\norm{\inv A}_{\mC _0} \leq \norm{\inv A}_{\mC _s} \leq  C_s \delta^{-2s - \frac{2}{s} -
  5}= \frac{1}{\epsilon} \, . 
$$
Now, since $\mC _{\nu _r} $ admits \nci\ in $\mC _0$, the explicit
version of Theorem~\ref{thm-dd-nci} in~\eqref{eq:asympexpa} implies that
$$
\norm{ \inv A} _{\mC _{\nu _r}} \leq \frac{1}{\epsilon} \, \nu _{r-1} \Big(
\frac{1}{\epsilon} \Big) = C_s \delta ^{-2s - \frac{2}{s} - 5} \, \nu
_{r-1} \Big(C_s \delta ^{-2s - \frac{2}{s} - 5} \Big) \, .
 $$
Now choose the  optimal value  $s=1$. 
\end{proof}
\begin{rem} For the commutative subalgebra of Toeplitz  matrices and the sub-exponential weight
  $v(k)=\exp(\abs k^{1/r})$, in ~\cite{elfalla98}  the
  following estimate was obtained: if  $\norm{A}_{\mC_{v}}\leq 1$ and $\norm {\inv
    A }_{\bop} \leq 1/\delta$, then 
\[
\norm{\inv A}_{\mC_{v}}  \leq C {\delta^{-2}} \log{\inv \delta} \exp(C
\delta^{-\frac{2}{r-1}}) \, .
\]
As so often, the technique for commutative Banach algebras is vastly
different from the non-commutative case and yields a better result.  
\end{rem}




\appendix

\section{ Iterated Quotient Rules}
\label{sec:combident}
In the following algebraic formulations we use the product 
$\prod _{i=1}^na_i = a_1 a_2 \dots a_n$ and should be aware  the
multiplication is in general  non commutative. 

\begin{lem} 
Let $\dd$ be a derivation on $\mA $ with $\one \in \mD (\dd )$ and $k\in \bn $. If $\inv a \in
\mD (\dd ^k)$, then 
    \begin{equation}\label{eq:itquota}
    \dd^k(\inv a)=\sum_{1 \leq m \leq k}(-1)^m \sum_{\substack{k_1+ \dotsb k_m=k \\ k_j \geq 1}} \binom{k}{k_1, \dotsc,k_m}
    \Bigl(\prod_{i=1}^m [\inv a \dd^{k_i}(a)]\Bigr) \inv a
  \end{equation}
\end{lem}
  \begin{proof}
Since $\one \in \mD (\dd )$, we have $\dd (\one) = 0$.    The proof
is by induction.
 For $k=1$ the assertion is true. The induction step $k-1 \to k$  is
 proved with the help of the iterated Leibniz rule 
$
\dd^k(\inv a a)=\sum_{j=0}^k \binom{k}{j} \dd^j(\inv a)
\dd^{k-j}(a) 
$ and $\dd ^k (e) = 0$. Using the induction hypothesis on $\dd
^j (\inv a )$ for $j<k$, we obtain 
\[
\begin{split}
\dd^k(\inv a) &=-\sum_{j=0}^{k-1} \binom{k}{j} \dd^j(\inv a) \dd^{k-j}(a) \, \inv a 
=-\inv a \dd^k(a) \inv a \\&- \sum_{j=1}^{k-1} \binom{k}{j} \sum_{m=1}^j(-1)^m \sum_{\substack{j_1+ \dotsb j_m=j \\ j_l \geq 1}} \binom{j}{j_1, \dotsc,j_m}
\Bigl(\prod_{l=1}^m [\inv a \dd^{j_l}(a)]\Bigr) \inv a \dd^{k-j}(a) \, \inv a.
\end{split}
\]
Interchanging the outer sums,  we then obtain
\[
\begin{split}
  \dd^k(\inv a)=&
-\inv a \dd^k(a) \inv a \\
&- \sum_{m=1}^{k-1}(-1)^m\sum_{j=m}^{k-1} \binom{k}{j}  \sum_{\substack{j_1+ \dotsb j_m=j \\ j_l \geq 1}} \binom{j}{j_1, \dotsc,j_m}
\Bigl(\prod_{l=1}^m [\inv a \dd^{j_l}(a)]\Bigr) \inv a \dd^{k-j}(a) \, \inv a\\
&=-\inv a \dd^k(a) \inv a
- \sum_{m=1}^{k-1}(-1)^m  \sum_{\substack{j_1+ \dotsb j_{m+1}=k \\ j_l \geq 1}} \binom{k}{j_1, \dotsc,j_{m+1}}
\Bigl(\prod_{l=1}^{m+1} [\inv a \dd^{j_l}(a)]\Bigr) \inv a \\
&=-\inv a \dd^k(a) \inv a
+ \sum_{m=2}^{k-1}(-1)^m  \sum_{\substack{j_1+ \dotsb j_{m}=k \\ j_l \geq 1}} \binom{k}{j_1, \dotsc,j_{m}}
\Bigl(\prod_{l=1}^{m} [\inv a \dd^{j_l}(a)]\Bigr) \inv a, 
\end{split}
\]
and this is (\ref{eq:itquot}).
  \end{proof}

\label{sec:iter-quot-rule}
\begin{lem}
  If $\Psi = (\psi_t) $ is an automorphism group  on  \mA, then the following
  iterated product and quotient rules are valid for all $a, b \in \mA
  $: 
  \begin{equation}
    \label{eq:dprodrule}
    \Delta_t^k(ab)=\sum_{l=0}^k \binom{k}{l} \psi_{(k-l)t}(\Delta_t^la)\Delta_t^{k-l}b
  \end{equation} 
  \begin{equation} 
    \label{eq:dquotrule}
    \Delta_t^k(\inv a)=\psi_{kt}(\inv a) \sum_{m=1}^k(-1)^m 
\sum_{\substack{k_1+\dotsc k_m=k\\ k_j\geq1}} \binom{k}{k_1,\dotsc,k_m} \prod_{j=1}^m \psi_{(k -\sum_{l=1}^j k_l)t}\bigl((\Delta_t^{k_j}a) \inv a \bigr)
  \end{equation}
\end{lem}
\begin{proof}
We start from the identity 
  \begin{equation}
    \label{eq:diffprod}
    \Delta_t(ab)= \psi_t(a) \Delta_t(b)+ \Delta_t(a) b \,,
  \end{equation}
which is \eqref{eq:dprodrule} for $k=1$ and can be verified by direct
computation. Now we  proceed by induction. If $k>1 , $ then
\begin{align*}
  \Delta_t^k(ab) &= \Delta_t^{k-1}(\Delta_t(ab))= \Delta_t^{k-1}\bigl(  \psi_t(a) \Delta_t(b)+ \Delta_t(a) b \bigr) \\
                 & =\sum_{l=0}^{k-1} \binom{k-1}{l} \psi_{(k-1-l)t}(\Delta_t^l(\psi_t a))\Delta_t^{k-1-l}(\Delta_t b) \\
                 &  +\sum_{l=0}^{k-1} \binom{k-1}{l} \psi_{(k-1-l)t}(\Delta_t^{l+1}a)\Delta_t^{k-1-l}b
\end{align*}
We obtain \eqref{eq:dprodrule} after re-indexing the second sum ($l \to
l+1$) and using an identity for binomial coefficients. 

 To prove the quotient rule for difference operators we start the induction with \eqref{eq:invderiv}, which is \eqref{eq:dquotrule} for $k=1$. The induction step for $k>1$  is similar to the proof of Lemma~\ref{invnormderiv}. As
\begin{equation*}
  0= \Delta_t^k (a \inv a)=\sum_{l=0}^k \binom{k}{l}
  \psi_{(k-l)t}(\Delta_t^la)\Delta_t^{k-l}(\inv a)=\sum_{l=0}^k
  \binom{k}{l} \psi_{lt}(\Delta_t^{k-l}a)\Delta_t^{l}(\inv a) \, ,
\end{equation*}
we obtain
\begin{equation*}
  \Delta_t^k(\inv a)=-\psi_{kt}(\inv a) \sum_{l=0}^{k-1} \binom{k}{l}
  \psi_{lt}(\Delta_t^{k-l}a)\Delta_t^{l}(\inv a) \, .
\end{equation*}
Using the induction hypotheses for $\Delta_t^{l}(\inv a)$ for $l<k$
in the last expression we proceed as follows: 
\begin{align*}
  &\sum_{l=0}^{k-1} \binom{k}{l}
  \psi_{lt}(\Delta_t^{k-l}a)\Delta_t^{l}(\inv a) = \Delta_t^k(a) \inv
  a \\ 
 & + \sum_{l=1}^{k-1} \binom{k}{l} \psi_{lt}(\Delta_t^{k-l}a)
               \psi_{ l t}(\inv a) \sum_{m=1}^l(-1)^m 
\sum_{\substack{l_1+\dotsc l_m=l\\ l_j\geq1}}
\binom{l}{l_1,\dotsc,l_m} \prod_{j=1}^m \psi_{(l -\sum_{r=1}^j
  l_r)t}\bigl(\Delta^{l_j}(a) \inv a \bigr) \\ 
&=  \Delta_t^k(a) \inv a 
+  \sum_{l=1}^{k-1} \sum_{m=1}^l (-1)^m \psi_{lt}\big(\Delta_t^{k-l}( a)
\inv a\big)\binom{k}{l}\sum_{\substack{l_1+\dotsc l_m=l\\ l_j\geq1}}
\binom{l}{l_1,\dotsc,l_m} \prod_{j=1}^m \psi_{(l -\sum_{r=1}^j
  l_r)t}\bigl((\Delta^{l_j}a) \inv a \bigr) \\ 
&=  \Delta_t^k(a) \inv a 
+  \sum_{m=1}^{k-1} (-1)^m\sum_{l=m}^{k-1} \,\sum_{\substack{l_1+\dotsc
    l_m=l\\ l_j\geq1}}\binom{k}{l} \binom{l}{l_1,\dotsc,l_m}
\psi_{lt}\big(\Delta_t^{k-l}( a)\inv a\big)\prod_{j=1}^m \psi_{(l -\sum_{r=1}^j
  l_r)t}\bigl((\Delta^{l_j}a) \inv a \bigr) \, . 
\end{align*}
Substituting $l_{m+1}= k-l$ the last line simplifies to
\begin{equation}
  \label{eq:dquotit}
   \Delta_t^k(a) \inv a 
+  \sum_{m=1}^{k-1} (-1)^m\sum_{\substack{l_1+\dotsc l_{m+1}=k\\
    l_j\geq1}} \binom{k}{l_1,\dotsc,l_{m+1}}  \prod_{j=1}^{m+1}
\psi_{(k -\sum_{r=1}^j l_r)t}\bigl((\Delta^{l_j}a) \inv a \bigr) \, ,
\end{equation}
and  we are done. 
\end{proof}

\def\cprime{$'$} \def\cprime{$'$} \def\cprime{$'$} \def\cprime{$'$}
  \def\cprime{$'$} \def\cprime{$'$} \def\cprime{$'$} \def\cprime{$'$}


\end{document}